\documentclass[12pt]{article}
\usepackage{array,amsmath}
\usepackage{amssymb}
\usepackage{graphicx,subfigure}
\usepackage{graphicx}
\usepackage{tikz}
\usetikzlibrary{shapes.geometric}
\makeatletter
\usepackage{fullpage}
\usepackage{amsthm}

\usepackage{fullpage}
\makeatother

\begin{document}
	
	\newtheorem{thm}{Theorem}[section]
	\newtheorem{lem}[thm]{Lemma}
	\newtheorem{con}[thm]{Conjecture}
	\newtheorem{cons}[thm]{Construction}
	\newtheorem{prop}[thm]{Proposition}
	\newtheorem{cor}[thm]{Corollary}
	\newtheorem{claim}[thm]{Claim}
	\newtheorem{obs}[thm]{Observation}
	\newtheorem{que}[thm]{Question}
	\newtheorem{defn}[thm]{Definition}
	\newtheorem{example}[thm]{Example}
	\newcommand{\di}{\displaystyle}
	\def\dfc{\mathrm{def}}
	\def\cF{{\cal F}}
	\def\cH{{\cal H}}
	\def\cK{{\cal K}}
	\def\cM{{\cal M}}
	\def\cA{{\cal A}}
	\def\cB{{\cal B}}
	\def\cG{{\cal G}}
	\def\cP{{\cal P}}
	\def\cC{{\cal C}}
	\def\ap{\alpha'}
	\def\Frk{F_k^{2r+1}}
	\def\nul{\varnothing} 
	\def\st{\colon\,}  
	\def\MAP#1#2#3{#1\colon\,#2\to#3}
	\def\VEC#1#2#3{#1_{#2},\ldots,#1_{#3}}
	\def\VECOP#1#2#3#4{#1_{#2}#4\cdots #4 #1_{#3}}
	\def\SE#1#2#3{\sum_{#1=#2}^{#3}}  \def\SGE#1#2{\sum_{#1\ge#2}}
	\def\PE#1#2#3{\prod_{#1=#2}^{#3}} \def\PGE#1#2{\prod_{#1\ge#2}}
	\def\UE#1#2#3{\bigcup_{#1=#2}^{#3}}
	\def\FR#1#2{\frac{#1}{#2}}
	\def\FL#1{\left\lfloor{#1}\right\rfloor} 
	\def\CL#1{\left\lceil{#1}\right\rceil}

	\title{The average connectivity matrix of a graph}
	\author{
		Linh Nguyen\thanks{Department of Applied Mathematics and Statistics, State University of New York, Stony Brook, NY 11794-3600,
			USA, linh.nguyen.1@stonybrook.edu}\,
		Suil O\thanks{Department of Applied Mathematics and Statistics, The State University of New York, Korea, Incheon, 21985, suil.o@sunykorea.ac.kr. Research supported by NRF-2020R1F1A1A01048226, NRF-2021K2A9A2A06044515, and NRF-2021K2A9A2A1110161711}\, 
	}
	
	\maketitle
	
	\begin{abstract}
		For a graph $G$ and for two distinct vertices $u$ and $v$, let $\kappa(u,v)$ be the maximum number of vertex-disjoint paths joining $u$ and $v$ in $G$.
		The average connectivity matrix of an $n$-vertex connected graph $G$, written $A_{\overline{\kappa}}(G)$, is an $n\times n$ matrix whose $(u,v)$-entry is $\kappa(u,v)/{n \choose 2}$ and let $\rho(A_{\overline{\kappa}}(G))$ be the spectral radius of $A_{\overline{\kappa}}(G)$. In this paper, we investigate some spectral properties of the matrix. In particular, we prove that for any $n$-vertex connected graph $G$,
		we have $\rho(A_{\overline{\kappa}}(G)) \le \frac{4\alpha'(G)}n$, which implies a result of Kim and O~\cite{KO} stating that for any connected graph $G$, we have $\overline{\kappa}(G) \le 2 \alpha'(G)$, where $\overline{\kappa}(G)=\sum_{u,v \in V(G)}\frac{\kappa(u,v)}{{n\choose 2}}$ and $\alpha'(G)$ is the maximum size of a matching in $G$; equality holds only when $G$ is a complete graph with an odd number of vertices. Also, for bipartite graphs, we improve the bound, namely $\rho(A_{\overline{\kappa}}(G)) \le \frac{(n-\alpha'(G))(4\alpha'(G) - 2)}{n(n-1)}$, and equality in the bound holds only when $G$ is a complete balanced bipartite graph.
	\end{abstract}
	
	\smallskip

	
	
	\section {Introduction}
	In this paper, we handle a simple, finite, and undirected graph.
	To check how well a graph $G$ is connected, we normally compute the \emph{connectivity} of $G$, written $\kappa(G)$, which is the minimum number of vertices $S$
	such that $G-S$ is disconnected or trivial.
	However, since this value is based on a worst-case situation, 
	Beineke, Oellermann, and Pippert~\cite{BOP} introduced a parameter to reflect a global amount of connectivity. The \emph{average connectivity} of $G$ is
	defined to be the number $\overline{\kappa}(G) = \sum_{u,v \in V(G)}\frac{\kappa(u,v)}{{n \choose 2}}$, where $\kappa(u,v)$ is the minimum number of vertices whose deletion makes $v$ unreachable from $u$. By Menger's Theorem, $\kappa(u,v)$ equals the maximum number of internally disjoint paths between $u$ and $v$.
	For convenience, we may assume that $\kappa(v,v)=0$ for every vertex $v \in V(G)$.
	
	The \emph{average connectivity matrix} of a graph $G$ with the vertex set $V(G)=\{v_1,\ldots,v_n\}$, written $A_{\overline{\kappa}}(G)$, is an $n\times n$ matrix whose $(i,j)$-entry is $\kappa(v_i,v_j)/{n\choose 2}$. For a vertex $v \in V(G)$, let the \emph{transmission of $v$}, written $T(v)$, be $\frac 1{{n\choose 2}}\sum_{w \in V(G)}\kappa(v,w)$ and let the \emph{transmission of $G$}, written $T(G)$, be $\max_{v \in V(G)} T(v)$. Note that for an $n$-vertex tree $T_n$, since there is a unique path between any $v_i$ and $v_j$, we have 
	${n\choose 2} A_{\overline{\kappa}}(T_n)=A(K_n) = J_n - I_n$,
	where $A(G)$ is the adjacency matrix of a graph $G$, $J_n$ is the $n\times n$ all-1 matrix, and $I_n$ is the $n\times n$ identity matrix.
	Note that $2\overline{\kappa}(G)= \sum_{v \in V(G)}T(v)$.
	
	\begin{obs}\label{obs}
		If $G$ is a connected graph with $n$ vertices, then for $v\in V(G)$, we have $\frac{2\kappa(G)}n \le T(v) \le \frac{2d(v)}n.$ 
	\end{obs}
	
	\begin{proof} For any pair of distinct vertices $u$ and $v$, we have $\kappa(G) \le \kappa(v,u) \le d(v)$, which implies
		that $(n-1)\kappa(G) \le {n\choose 2} T(v) \le (n-1)d(v)$.
	\end{proof}
	Equalities in the bounds in Observation~\ref{obs} hold for complete graphs. For an $n$-vertex tree, equality in the lower bound holds, but for the center vertex $v$ in an $n$-vertex star, there is a big gap between $T(v)$ and $2d(v)/n$.
	
	A connected graph $G$ is {\it uniformly $r$-connected} 
	if for any pair of non-adjacent vertices $u$ and $v$, we have $\kappa(u,v)=r$. Note that if $G$ is uniformly $r$-connected, then the average connectivity matrix has constant row sum equal to $2r/n$.
	
	Since $A_{\overline{\kappa}}(G)$ is real and symmetric, all of its eigenvalues are real. Thus we can index the eigenvalues, say $\lambda_1(A_{\overline{\kappa}}(G)),\ldots,\lambda_n(A_{\overline{\kappa}}(G))$, in non-increasing order.  For simplicity, we let $\lambda_i(A_{\overline{\kappa}}(G))=\lambda_i$, and 
	let $\rho(A_{\overline{\kappa}}(G))=\max\{|\lambda_i|:1\le i \le n\}$. If we let $\rho(A_{\overline{\kappa}}(G))=\rho$, then by Theorem~\ref{PF}, we have $\lambda_1=\rho$.
	We prove some relationships between the spectral radius and some quantities of the graph.
	
	\begin{thm} \label{basic}
		For an $n$-vertex connected graph $G$, we have $\frac{2\overline{\kappa}(G)}n \le \rho\le T(G)$.
		Equalities hold only when $G$ is uniformly connected.
	\end{thm}
	
	\begin{proof} By Theorem~\ref{PF}, there exists a unique positive eigenvector $\bold{x}$ corresponding to $\rho$. Let $x_i=\max_{j=1,\ldots,n} x_j$. Then we have
		$$\rho x_i=\sum_{j=1}^n(A_{\overline{\kappa}}(G))_{ij}x_j\le x_i\sum_{j=1}^n(A_{\overline{\kappa}}(G))_{ij} = T(v_i)x_i \le T(G)x_i,$$
		which implies that $\rho \le T(G)$. Equality holds only when $G$ is a uniformly connected graph.
		
		For the lower bound, by Theorem~\ref{RR}, we have 
		$$ \frac{2 \overline{\kappa}(G)}n =  \frac{{\bf 1}^TA_{\overline{\kappa}}(G){\bf 1}}{{\bf 1}^T{\bf 1}} \le \rho.$$
		Equality holds only when $\bold{x}=c{\bf 1}$ for some constant $c$, i.e. $G$ is uniformly connected.
	\end{proof}
	
	A \emph{matching} in a graph $G$ is a set of disjoint edges in it.
	The \emph{matching number} of a graph $G$, written $\alpha'(G)$,
	is the maximum size of a matching in it.
	Kim and O~\cite{KO} gave a relation between the matching number and the average connectivity of $G$.
	
	\begin{thm}{\rm \cite{KO}}\label{KO}
		For a connected graph $G$, we have
		$\overline{\kappa}(G) \le 2\alpha'(G).$
	\end{thm}
	
	In Section 3, we prove an upper bound for the spectral radius of $A_{\overline{\kappa}}(G)$ in terms of the matching number and the number of vertices.
	
	\begin{thm}\label{main}
		For an $n$-vertex connected graph $G$, we have $$\rho(A_{\overline{\kappa}}(G)) \le \frac{4\alpha'(G)}n.$$
		Equality holds only when $G=K_{n}$ and $n$ is odd.
	\end{thm}
	
	If Theorem~\ref{main} is true, then by Theorem~\ref{basic}, we have $\frac{2\overline{\kappa}(G)}n \le \rho \le \frac{4\alpha'(G)}n$, which implies Theorem~\ref{KO}. To prove Theorem \ref{main}, we find the maximum spectral radius among $n$-vertex connected graphs with a given matching number. However, this idea does not derive the relation between $\overline{\kappa}(G)$ and matching number $\alpha'(G)$ in connected bipartite graphs by Kim and O (see theorem 2.3 in \cite{KO}). Nevertheless, we also prove an upper bound for the spectral radius of $A_{\overline{\kappa}}(G)$ in $n$-vertex connected bipartite graphs to guarantee a certain matching number in Section 4.
	
	\begin{thm}\label{bipartitethm}
		For an $n$-vertex connected bipartite graph $G$, we have $$\rho(A_{\overline{\kappa}}(G)) \le \frac{(n-\alpha'(G))(4\alpha'(G) - 2)}{n(n-1)}.$$
		Equality holds only when $G=K_{n/2,n/2}$.
	\end{thm}
	
	For undefined terms of graph theory, see West~\cite{W}. For basic properties of spectral graph theory, see Brouwer and Haemers~\cite{BH} or  Godsil and Royle~\cite{GR}.\\
	
	\section{Tools} 
	
	In this section, we provide the main tools to prove Theorem \ref{main} and Theorem \ref{bipartitethm}.
	

	
	
	\begin{thm} {\rm (Rayleigh-Ritz Theorem; see~{\cite[Theorem~4.2.2]{HJ}})}\label{RR}
		If $A$ is an $n\times n$ Hermitian matrix, then 
		$$\rho(A)=\max_{x \neq {\bf 0}}\frac{x^*Ax}{x^*x}.$$
	\end{thm}
	Theorem~\ref{RR} is used to prove Theorem~\ref{basic}. 
	
	The Perron-Frobenius Theorem is a very important theorem, implying that $\rho=\lambda_1$ and also the Perron vector is positive.
	
	\begin{thm} {\rm (Perron-Frobenius Theorem; see~{\cite[Theorem~8.4.4]{HJ}})}\label{PF}
		Let $A$ be an $n\times n$ matrix and suppose that $A$ is irreducible and non-negative. Then\\
		(a) $\rho(A) > 0;$\\
		(b) $\rho(A)$ is an eigenvalue of $A$; \\
		(c) There is a positive vector $x$ such that $Ax=\rho(A)x$; and\\
		(d) $\rho(A)$ is an algebraically (and hence geometrically) simple eigenvalue of $A$.
	\end{thm}
	
	To compare the spectral radii of two average connectivity matrices, we often need Theorem~\ref{PerronCor}.
	
	
		

	
	\begin{thm}{\rm See~{\cite[Theorem~8.4.5]{HJ}})\label{PerronCor}
			Let $A$ be an $n\times n$ non-negative and irreducible matrix. If $A \ge |B|$, then we have $\rho(A) \ge \rho(B)$.}
	\end{thm}
	
	If we encounter a problem when we compare the spectral radii of two matrices, then we normally compare the spectral radii of matrices (called quotient matrices) derived from vertex partitions.
	
	
	%
	
	
	Let $A$ be a real symmetric matrix of order $n$, whose rows and columns are indexed by $P=\{1,\ldots,n\}$, and let $\{P_1,\ldots,P_q\}$ be a partition of $P$.
	If for each $i \in \{1,\ldots,q\}$, $n_i=|P_i|$, then $n=n_1+\cdots +n_q$. Let $A_{i,j}$ be the submatrix of $A$ formed by rows in $P_i$ and columns in $P_j$.
	The \emph{quotient matrix} of $A$ with respect to the partition $P$ is the $q\times q$ matrix $Q=(q_{i,j})$, where $q_{i,j}$ is the average row sum of $A_{i,j}$.
	If the row sum of each matrix $A_{i,j}$ is constant, then we call the partition $P$ \emph{equitable}.
	\begin{thm}{\rm \cite{YYSX}}\label{quotient}
		If a partition $P$ corresponding to the quotient matrix $Q$ of $A$ is equitable, then we have $\rho(Q)=\rho(A)$.
	\end{thm}
	
	\begin{thm}{\rm \cite[Theorem~6.1.1]{HJ} (Gershgorin circle theorem)}\label{Gctheo} Let $A$ be an $n$ by $n$, and let
		\[R'_i(A) \equiv \sum_{j=1,j\ne i}^{n}|a_{ij}|, \quad 1 \le i \le n\]
		denote the deleted absolute row sums of $A$. Then all the eigenvalues of $A$ are located in the union of $n$ discs
		
		\[\bigcup_{i=1}^n\{z\in\mathbb{C}: |z - a_{ii}| \le R'_i(A) \}. \]
	\end{thm}
	
	Since the diagonal entries of $A_{\overline{\kappa}}(G)$ are all zeros, the Gershgorin circle theorem says that $\rho\left(A_{\overline{\kappa}}(G)\right)$ is bounded by the maximum row sum. Theorem \ref{Gctheo} is used to prove Theorem \ref{bipartitethm}.
	
	\begin{thm}{\rm \cite{BT} (Berge-Tutte formula)}\label{TBform}
		For an $n$-vertex graph $G$, we have 
		
		\[\alpha'(G)=\min\limits_{S\subseteq V(G)}\frac{1}{2}\left(n - o(G-S) + |S|\right),\]
		where $o(H)$ is the number of components with odd number of vertices in a graph $H$.
	\end{thm}
	
	\section{Proof of Theorem~\ref{main}}
	In this section, we prove Theorem \ref{main} by using the tools in the previous section.
	
	For a vertex subset $S\subseteq V(G)$, let $\dfc(S)=o(G-S)-|S|$, and let $\dfc(G)=\max_{S\subseteq V(G)}\dfc(S)$. For two disjoint graphs $G_1$ and $G_2$, the \emph{join graph} $G_1 \vee G_2$ is the graph with $V(G_1\vee V_2)=V(G_1)\cup V(G_2)$ and $E(G)=E(G_1)\cup E(G_2) \cup \{uv: u \in V(G_1) \text{ and } v \in V(G_2)\}$. 
	
	Now we are ready to prove Theorem \ref{main}. The idea basically comes from the paper~\cite{O2}, and many researchers~\cite{KOSS,LM,Z,ZHW,ZL} used it to prove upper bounds for the spectral radius of many types of matrices in an $n$-vertex ($t$-)connected graph.
	
	\quad\\
	\textit{Proof of theorem \ref{main}.} First, we handle when $G$ has a (near) perfect matching, and then in the remaining case, we use the Berge-Tutte Formula.
	
	\begin{claim}\label{claim1}
		For an $n$-vertex connected graph $G$,
		we have $\rho \le \frac{2(n-1)}n.$
	\end{claim}
	\begin{proof}
		By Theorem~\ref{PerronCor}, we have $\rho(G) \le \rho(K_n)=\frac 2{n(n-1)}(n-1)^2=\frac{2(n-1)}n$.
	\end{proof}
	
	If $G$ has a perfect matching $(\alpha'(G) = \frac{n}{2})$ or a near perfect matching $(\alpha'(G) = \frac{n-1}{2})$,
	then by Claim~\ref{claim1}, we have $$\rho \le \frac {4\alpha'(G)}n.$$ 
	
	Now, we may assume that for $t \ge 2$, $\alpha'(G) = \frac{n-t}2$. Among $n$-vertex connected graphs $H$ with $\alpha'(H) = \frac{n-t}2$, let $G$ be a graph such that $\rho(G) \ge \rho(H)$.
	By Theorem~\ref{TBform}, there exists a vertex subset $S$ such that
	$\dfc(S) \ge 2$. Let $o(G-S)=q$ and $|S|=s$. Thus $t=\dfc(G) = q-s \ge 2$.
	
	Now, we want to show that $\rho\le\frac {2(n-t)}n$.
	
	\begin{claim}\label{mainclaim}
		If $G$ is an $n$-vertex connected graph with $\alpha'(G)=\frac{n-t}2$, where $t \ge 2$, then $\rho(G) \le \rho(K_{s} \vee \left(K_{n-s-q+1}\cup (q-1)K_1 \right)).$
	\end{claim}
	
	\begin{proof}
		Note that $G$ does not contain even components in $G-S$. 
		If there are at least two even components in $G-S$, then the graph $G'$ obtained from $G$ by joining two even components in $G-S$ has a bigger spectral radius, which is a contradiction by the choice of $G$.
		If there is exactly one even component in $G-S$, then the graph $G''$ obtained from $G$ by joining the even component and an odd component has a bigger spectral radius, which is a contradiction.
		Thus we say that there are only odd components of $G-S$, say $G_1,\ldots, G_q$, and for each $i \in \{1,\ldots, q\}$, let $|V(G_i)|=n_i$, where $n_1 \ge \ldots \ge n_q$. 
		Note that $n=s+n_1+\cdots+n_q$.
		By Theorem~\ref{PerronCor}, we have
		$$\rho(A_{\overline{\kappa}}(G)) \le \rho\left(A_{\overline{\kappa}}(K_s \vee (K_{n_1}\cup \ldots \cup K_{n_q})\right),$$
		and note that $\alpha'(G)=\alpha'\left(K_s \vee (K_{n_1}\cup \ldots \cup K_{n_q})\right)$.
		
		Now, if $n_q \ge 3$, then we show that 
		\[ \rho\left(A_{\overline{\kappa}}(K_s \vee (K_{n_1}\cup \ldots \cup K_{n_q})\right) \le \rho\left(A_{\overline{\kappa}}(K_s \vee (K_{n_1+2}\cup \ldots \cup K_{n_q-2})\right).\]
		Let $Q$ be the quotient matrix corresponding to the vertex partition $P=\{V(K_s),V(K_{n_1}),\ldots,V(K_{n_q})\}$ of $K_s \vee \left(K_{n_1}\cup \ldots \cup K_{n_q}\right)$. Then for $1 \le i, j \le n$, we have
		
		\begin{align*}
			\begin{pmatrix}n\\2\end{pmatrix}Q_{11}&=(n-1)(s-1)\\
			\begin{pmatrix}n\\2\end{pmatrix}Q_{ii}&=(s+n_{i-1} -1) \times (n_{i-1} -1)\quad (i =2,\ldots, n)\\
			\begin{pmatrix}n\\2\end{pmatrix}Q_{ij}&=sn_{j-1} \quad (i \neq j)\\
			\begin{pmatrix}n\\2\end{pmatrix}Q_{1i}&=(s+n_{i-1} -1)n_{i-1}\quad (i\neq1)\\
			\begin{pmatrix}n\\2\end{pmatrix}Q_{i1} &= (s+n_{i-1} -1)s\quad (i\neq1).
		\end{align*}
		
		Let $x=(y,x_1,\ldots,x_q)$ be the Perron vector corresponding to $\rho(Q)$, and also note that $x_1 \ge \ldots \ge x_q$ since 
		$${n\choose 2}\rho x_i = (s + n_i -1)sy + sn_1x_1 + ... +sn_{i-1}x_{i-1} + (s + n_i -1)(n_i - 1)x_i + sn_{i+1}x_{i+1} + ... + sn_qx_q,$$
		which implies
		$${n\choose 2}(\rho x_i - \rho x_j) = (n_i - n_j)sy + (n_i - 1)^2 x_i - (n_j-1)^2 x_j - s(x_i - x_j).$$
		$$\Rightarrow\left(\rho + \frac 1{{n\choose 2}}s\right)(x_i - x_j) = (n_i - n_j)sy + (n_i - 1)^2 x_i - (n_j-1)^2 x_j$$ $$\ge (n_j - 1)^2(x_i - x_j) 
		\ge (n_j - 1)(x_i - x_j).$$
		Since $\rho > n_j -1$ and $n_j \ge 3$, we have $x_i \ge x_j$ for $i \le j$.
		
		Now we compare the spectral radii of $A_{\overline{\kappa}}\left(K_s \vee (K_{n_1+2}\cup \ldots \cup K_{n_q-2})\right)$ and $A_{\overline{\kappa}}\left(K_s \vee (K_{n_1}\cup \ldots \cup K_{n_q})\right)$. If $Q'$ is the quotient matrix corresponding to the vertex partition $P'=\{V(K_s),V(K_{n_1+2}),$ $\ldots,V(K_{n_q-2})\}$ of $K_s \vee \left(K_{n_1+2}\cup \ldots \cup K_{n_q-2}\right)$,
		then we have
		$$x^T(Q'-Q)x=(4s+4n_1+2)y x_1+(-4s-4n_q+6)y x_q+(4n_1)x_1^2+(-4n_q+8)x_q^2+2s(x_1-x_q)
		\sum_{i=1}^{q} x_i$$
		$$=2s(x_1-x_q)(2y+\sum_{i=1}^{q} x_i)+4y(n_1 x_1-n_q x_q)+4(n_1x_1^2-n_q x_q^2)+2y x_1+6y x_q+8x_q^2>0
		$$
		since $x_1 \ge \ldots \ge x_q > 0$. Therefore
		
		\[\rho\left(A_{\overline{\kappa}}(K_s \vee (K_{n_1}\cup \ldots \cup K_{n_q}))\right) \le \rho\left(A_{\overline{\kappa}}(K_s \vee (K_{n_1+2}\cup \ldots \cup K_{n_q-2}))\right).\]
		Essentially, moving 2 vertices from $K_{n_q}$ to $K_{n_1}$ increases the spectral radius. Using a similar comparison approach, we can prove that moving 2 vertices from any $K_{n_i}, i = 2, \ldots, n_q$ to $K_{n_1}$ also increases the spectral radius. Thus, we can keep transferring vertices, two at a time, to $K_{n_1}$ until $n_2 = \ldots = n_q = 1$ to obtain $K_s\vee(K_{n-s-q+1}\cup (q-1)K_1)$, and clearly
		
		\[\rho\left(A_{\overline{\kappa}}(K_s \vee (K_{n_1+2}\cup \ldots \cup K_{n_q-2}))\right) \le \rho\left(A_{\overline{\kappa}}(K_s\vee(K_{n-s-q+1}\cup (q-1)K_1))\right),\]
		which completes the proof of Claim \ref{mainclaim}.
	\end{proof}
	
	The next step is to handle when $n_1 = n-s-q+1, n_2 = \ldots =n_q = 1$. Depending on $n$ relatively to $t$, we consider two cases because we have two different types of extremal graphs. Also, we will denote the extremal graphs in terms of $t$, namely $K_{s} \vee \left(K_{n-s-q+1}\cup (q-1)K_1 \right)$ becomes $K_s\vee({K_{n-2s-t+1}\cup \overline{K_{t+s-1}}})$.
	
	\begin{claim}\label{extremalclaim1}
		If  $n \ge 3t + 2$, then $\rho\left(A_{\overline{\kappa}}(K_s\vee({K_{n-2s-t+1}\cup \overline{K_{t+s-1}}}))\right) \le \rho\left(A_{\overline{\kappa}}(G_1)\right)$, where $G_1 = K_1\vee({K_{n-t-1}\cup \overline{K_{t}}}).$
	\end{claim}
	
	\begin{proof}
		The quotient matrix $Q_0$ corresponding to the vertex partition $\{V(K_s),V(K_{n-2s-t+1}),\allowbreak V(\overline{K_{t+s-1}})\}$ is
		
		\[\begin{pmatrix}n\\2\end{pmatrix}Q_0 = \begin{pmatrix}
			(n-1)(s-1) & (n-s-t)(n-2s-t+1)&s(s+t-1)\\
			(n-s-t)s & (n-s-t)(n-2s-t)&s(s+t-1)\\
			s^2&s(n-2s-t+1)&s(s+t-2)
		\end{pmatrix}\]
		and the quotient matrix $Q_1$ corresponding to the vertex partition $\{V(K_1),V(K_{n-1}),\allowbreak V(\overline{K_{t}})\}$ is
		
		\[\begin{pmatrix}n\\2\end{pmatrix}Q_1 = \begin{pmatrix}
			0&(n-t-1)^2&t\\
			n-t-1&(n-t-1)(n-t-2)&t\\
			1&n-t-1&t-1
		\end{pmatrix}.\] 
		
		\quad
		
		\noindent\textit{Case 1: }If $s \le \frac{n-t}{2} - 1$, let $f(\lambda)$ be the characteristic polynomial of $\begin{pmatrix}n\\2\end{pmatrix}Q_0$ and $f_1(\lambda)$ be the characteristic polynomial of $\begin{pmatrix}n\\2\end{pmatrix}Q_1$. Suppose that the coefficient of $\lambda^3$ in each $f(\lambda)$ and $f_1(\lambda)$ is positive without loss of generality.
		
		If $\theta_1$ is the largest root of $f_1(\lambda)$, since $K_{n-t}$ is a subgraph of $G_1$
		\[\theta_1 \ge \rho\left({n\choose 2}A_{\overline{\kappa}}(K_{n-t})\right) = (n-t-1)^2.\]
		
		As a result, if $f((n-t-1)^2) > 0$, which implies $(n-t-1)^2$ is larger than the largest root of $f$, then $\rho(A_{\overline{\kappa}}(G_1))\ge\rho(A_{\overline{\kappa}}(G))$. Let $g(\lambda) = f(\lambda) - f_1(\lambda)$. We shall prove that $f((n-t-1)^2) > 0$ by proving that $g((n-t-1)^2) + f_1((n-t-1)^2) > 0$.
		Given below is the explicit formula for $g(\lambda)$:
		
		\begin{align*}
			&g(\lambda) =\lambda^2(s - 1)(2n - 4t - 3s) +\lambda(s - 1)(2t^3 - 4t^2n + 9t^2s + 5t^2 + 2tn^2 - 10tns - 6tn + 12ts^2\\& - 2ts + 2t+ 2n^2s + n^2 - 6ns^2+  ns +  5s^3 - 5s^2 + s - 1)	-(s - 1)(t^4s + t^4 - 2t^3ns - 3t^3n\\& + 6t^3s^2 + 2t^3s + 3t^3 + t^2n^2s + 3t^2n^2 - 8t^2ns^2 - 5t^2ns - 5t^2n + 13t^2s^3 - 6t^2s^2 + 3t^2s + t^2 - tn^3\\& + 2tn^2s^2 + 4tn^2s + tn^2 - 10tns^3 + 4tns^2 - 2tns	 + 2tn + 12ts^4 - 16ts^3 + 6ts^2 - 2ts - 2t - n^3s +\\& n^3 + n^2s^3 + n^2s^2 - n^2s - 3n^2 - 4ns^4 + 6ns^3 - 3ns^2 + 3ns + 3n + 4s^5 - 9s^4 + 6s^3 - 2s^2 - s - 1).
		\end{align*}
		
		If $s = 1$, then $g(\lambda) = 0$ and $G\cong G_1$, the claim becomes trivial, so we restrict $s\ge 2$. Let $\tilde{g}(\lambda) = \frac{g(\lambda)}{s - 1}$, then $g(\lambda) \ge \tilde{g}(\lambda)$. Consider $\Tilde{g}((n-t-1)^2)$ as a function $h(s)$ of $s$. We can prove that $h(s) \ge h\left(\frac{n-t}{2} - 1\right)$ for all $s \le \frac{n-t}{2} - 1$. Indeed
		
		\begin{align*}
			&h(s) - h\left(\frac{n-t}{2}-1\right) = \frac{11n^5}{8} - n^4s - \frac{43n^4t}{8} - \frac{133n^4}{16} - 6n^3s^2 - 2n^3st + 10n^3s + \frac{29n^3t^2}{4}\\& + \frac{63n^3t}{2} + \frac{111n^3}{8} + 4n^2s^3 + 22n^2s^2t + 6n^2s^2 + 12n^2st^2 - 20n^2st - 16n^2s - \frac{13n^2t^3}{4} - \frac{343n^2t^2}{8}\\&  - \frac{375n^2t}{8} - \frac{13n^2}{2} + 4ns^4 - 16ns^3 - 22ns^2t^2 - 30ns^2t + 7ns^2 - 14nst^3 + 8nst^2 + 32nst + 8ns - \\& \frac{5nt^4}{8} + \frac{49nt^3}{2} + \frac{421nt^2}{8} + 12nt + 8n - 4s^5 - 12s^4t + 9s^4 - 8s^3t^2 + 26s^3t - s^3 + 6s^2t^3 + 25s^2t^2\\& - 4s^2t - 3s^2 + 5st^4 + 2st^3 - 15st^2 - 10st - s + \frac{5t^5}{8} - \frac{77t^4}{16} - \frac{157t^3}{8} - \frac{19t^2}{2} + 5t - 12
		\end{align*}
		
		\begin{align*}
			=&\left(\frac{n-t}{2} - s - 1\right)\left(\frac{11n^4}4 + \frac{7n^3s}2 - 8n^3t - \frac{89n^3}8 - 5n^2s^2 - \frac{33n^2st}2 + \frac{19n^2s}4 + \frac{13n^2t^2}2\right.\\
			&+ \frac{287n^2t}8 + \frac{11n^2}2 - 2ns^3 + 6ns^2t + \frac{23ns^2}2 + \frac{41nst^2}2 + \frac{7nst}2 - \frac{23ns}2 - \frac{295nt^2}8 - \frac{33nt}2\\
			&- 2n + 4s^4 + 10s^3t - 13s^3 + 3s^2t^2 - \frac{59s^2t}2 + 14s^2 - \frac{15st^3}2 - \frac{53st^2}4 + \frac{53st}2 - 11s - \frac{5t^4}4\\&\left.\vphantom{\frac12} + \frac{97t^3}8 + 15t^2 - 11t + 12\right)
		\end{align*}
		\begin{multline*}
			=\left(\frac{n-t}{2} - s - 1\right)A.\hfill
		\end{multline*}
		
		Now we prove that $A > 0$. We abuse notions slightly by multiplying $A$ by 8 to make all coefficients whole numbers, which makes for a cleaner presentation but still denoting the resulting sum by $A$.
		\begin{align*}
			A =& 22n^4 + 28n^3s - 64n^3t - 89n^3 - 40n^2s^2 - 132n^2st + 38n^2s + 52n^2t^2 + 287n^2t + 44n^2\\& - 16ns^3 + 48ns^2t + 92ns^2 + 164nst^2 + 28nst - 92ns - 295nt^2 - 132nt - 16n + 32s^4\\& + 80s^3t - 104s^3 + 24s^2t^2 - 236s^2t + 112s^2 - 60st^3 - 106st^2 + 212st - 88s - 10t^4\\& + 97t^3 + 120t^2 - 88t + 96.
		\end{align*}
		
		\begin{align*}
			&\frac{\partial A}{\partial n} = 88n^3 + 84n^2s - 192n^2t - 267n^2 - 80ns^2 - 264nst + 76ns + 104nt^2 + 574nt + 88n - 16s^3,\\
			& + 48s^2t + 92s^2 + 164st^2 + 28st - 92s - 295t^2 - 132t - 16,\\
			&\frac{\partial^2 A}{\partial n^2} = 264n^2 + 168ns - 384nt - 534n - 80s^2 - 264st + 76s + 104t^2 + 574t + 88,\\
			&\frac{\partial^3 A}{\partial n^3} = 528n + 168s - 384t - 534.
		\end{align*}
		From $n\ge3t+2$ and $n\ge t + 2s + 2$, we have $n\ge 2t + s + 2$ and so
		
		\begin{align*}
			&\frac{\partial^3 A}{\partial n^3} > 0,\\
			&\frac{\partial^2 A}{\partial n^2} \ge \frac{\partial^2 A}{\partial n^2}(2t + s + 2) = 352s^2 + 744st + 934s + 392t^2 + 850t + 76 > 0,\\
			&\frac{\partial A}{\partial n} \ge \frac{\partial A}{\partial n}(2t + s + 2) = 76s^3 + 296s^2t + 605s^2 + 364st^2 + 1174st + 472s + 144t^3 + 569t^2 +\\
			& 400t - 204 > 0,\\
			&A\ge A(2t + s + 1) = 26s^4 + 84s^3t + 89s^3 + 128s^2t^2 + 473s^2t + 570s^2 + 108st^3 + 615st^2 + 932st\\
			& - 100s + 38t^4 + 231t^3 + 386t^2 - 124t - 120 > 0.
		\end{align*}
		
		Therefore, $A > 0$, and consequently
		\begin{align*}
			&\tilde{g}((n-t-1)^2) + f_1((n-t-1)^2) \ge h\left(\frac{n-t}{2}-1\right) + f_1((n-t-1)^2) = \frac{5n^5}8 - \frac{37n^4t}8 + \frac{21n^4}{16}\\& + \frac{51n^3t^2}4 - \frac{7n^3t}2 - \frac{39n^3}8 - \frac{67n^2t^3}4 + \frac{7n^2t^2}8 + \frac{135n^2t}8 + \frac{3n^2}2 + \frac{85nt^4}8 + \frac{7nt^3}2 - \frac{157nt^2}8 + 2nt\\& - 7n - \frac{21t^5}8 - \frac{35t^4}{16} + \frac{61t^3}8 + \frac{t^2}2 - 7t + 12.
		\end{align*}
		
		Consider the right hand side of the above inequality a function of $n$
		\begin{align*}
			&RHS'(n) = \frac{25n^4}8 - \frac{37n^3t}2 + \frac{21n^3}4 + \frac{153n^2t^2}4 - \frac{21n^2t}2 - \frac{117n^2}8 - \frac{67nt^3}2 + \frac{7nt^2}4 + \frac{135nt}4 + 3n\\& + \frac{85t^4}8 + \frac{7t^3}2 - \frac{157t^2}8 + 2t - 7,\\
			&RHS''(n) = \frac{25n^3}2 - \frac{111n^2t}2 + \frac{63n^2}4 + \frac{153nt^2}2 - 21nt - \frac{117n}4 - \frac{67t^3}2 + \frac{7t^2}4 + \frac{135t}4 + 3,\\
			&RHS^{(3)}(n) = \frac{75n^2}2 - 111nt + \frac{63n}2 + \frac{153t^2}2 - 21t - \frac{117}4,\\
			&RHS^{(4)}(n) = 75n - 111t + \frac{63}2 > 0.
		\end{align*}
		Therefore, since $n\ge 3t + 2$
		\begin{align*}
			&RHS^{(3)}(n)\ge RHS^{(3)}(3t + 2) = 81t^2 + \frac{603t}2 + \frac{735}4> 0,\\
			&RHS''(n)\ge RHS''(3t + 2) = 34t^3 + \frac{485t^2}2 + 321t + \frac{215}2> 0,\\
			&RHS'(n)\ge RHS'(3t + 2) = 8t^4 + 124t^3 + 273t^2 + 202t + \frac{65}2> 0,\\
			&RHS(n)\ge RHS(3t + 2) = 44t^4 + 149t^3 + 189t^2 + 60t + 6 > 0.
		\end{align*}
		To sum up, $0 < \tilde{g}((n-t-1)^2) + f_1((n-t-1)^2) \le g((n-t-1)^2) + f_1((n-t-1)^2) = f((n-t-1)^2)$.
		
		\quad
		
		\noindent\textit{Case 2: }If $s = \frac{n-t}{2}$, the largest root of $f(\lambda)$ can be computed explicitly as
		\begin{align*}
			&\rho(f(\lambda)) = - \frac{nt}4 + \frac{3n^2}8 - \frac{t^2}8 + \frac12 
			+ \frac{t}2 - n + \\&\frac{1}{8}\sqrt{5n^4 - 12n^3t - 8n^3 + 6n^2t^2 + 16n^2t + 24n^2 + 4nt^3 - 8nt^2 - 16nt - 32n - 3t^4 + 8t^2 + 16}.
		\end{align*}
		
		We can prove that $\rho(f(\lambda)) \le (n-t-1)^2$ directly by considering
		\begin{align*}
			& 8(n-t-1)^2 - 8\rho(f(\lambda)) = 12t - 8n - 14nt + 5n^2 + 9t^2 + 4 - \\&\sqrt{5n^4 - 12n^3t - 8n^3 + 6n^2t^2 + 16n^2t + 24n^2 + 4nt^3 - 8nt^2 - 16nt - 32n - 3t^4 + 8t^2 + 16}.
		\end{align*}
		Since $12t - 8n - 14nt + 5n^2 + 9t^2 + 4 > 0$, if we can prove that $(12t - 8n - 14nt + 5n^2 + 9t^2 + 4)^2 - (5n^4 - 12n^3t - 8n^3 + 6n^2t^2 + 16n^2t + 24n^2 + 4nt^3 - 8nt^2 - 16nt - 32n - 3t^4 + 8t^2 + 16) \ge 0$, we will have proved that $(n-t-1)^2 \ge \rho(A_{\overline{\kappa}}(G))$. Indeed
		\begin{align*}
			&(12t - 8n - 14nt + 5n^2 + 9t^2 + 4)^2 - (5n^4 - 12n^3t - 8n^3 + 6n^2t^2 + 16n^2t + 24n^2 + 4nt^3 - 8nt^2\\
			&- 16nt - 32n - 3t^4 + 8t^2 + 16)\\
			=&20n^4 - 128n^3t - 72n^3 + 280n^2t^2 + 328n^2t + 80n^2 - 256nt^3 - 472nt^2 - 288nt - 32n + 84t^4 + 216t^3 \\&+ 208t^2 + 96t
		\end{align*}
		
		Consider the right hand side of the above equality a function of $n$
		\begin{align*}
			&RHS'(n) = 80n^3 - 384n^2t - 216n^2 + 560nt^2 + 656nt + 160n - 256t^3 - 472t^2 - 288t - 32,\\
			&RHS''(n) = 240n^2 - 768nt - 432n + 560t^2 + 656t + 160\\
			&RHS^{(3)}(n) = 480n - 768t - 432. > 0
		\end{align*}
		\begin{align*}
			\Rightarrow&RHS''(n)\ge RHS''(3t + 2) = 416t^2 + 704t + 256> 0,\\
			&RHS'(n)\ge RHS'(3t + 2) = 128t^3 + 384t^2 + 256t + 64> 0,\\
			&RHS(n)\ge RHS(3t + 2) = 64t^3 > 0.
		\end{align*}
		
		Hence, for any value of $s$, if $n\ge 3t + 2$ then $\rho(A_{\overline{\kappa}}(G))\le\frac{1}{\begin{pmatrix}
				n\\2
		\end{pmatrix}}(n-t-1)^2\le\rho(A_{\overline{\kappa}}(G_1))$.
		
	\end{proof}
	
	\begin{claim}\label{extremalclaim2}
		If  $n \le 3t$, then $\rho(A_{\overline{\kappa}}(K_s\vee({K_{n-2s-t+1}\cup \overline{K_{t+s-1}}}))) \le \rho(A_{\overline{\kappa}}(G_2))$, where $G_2 = K_{\frac{n-t}{2}}\vee{\overline{K_{\frac{n+t}{2}}}}.$
	\end{claim}
	
	\begin{proof}
		The quotient matrix $Q_2$ corresponding to the vertex partition $\left\{V(K_{\frac{n-t}{2}}),V(\overline{K_{\frac{n+t}{2}}})\right\}$ is
		
		\[\begin{pmatrix}n\\2\end{pmatrix}Q_2 = \begin{pmatrix}\frac{(n-1)(n-t-2)}{2}&\frac{n^2-t^2}{4}\\\frac{(n-t)^2}{4}&\frac{(n-t)(n+t-2)}{4}\end{pmatrix}.\]
		
		Let $f(\lambda)$ be the characteristic polynomial of $\begin{pmatrix}n\\2\end{pmatrix}Q_0$ defined in Claim \ref{extremalclaim1} and $f_2(\lambda)$ be the characteristic polynomial of $\begin{pmatrix}n\\2\end{pmatrix}Q_2$. If $\theta_2$ is the largest root of $f_2(\lambda)$ and it's larger than that of $f(\lambda)$, then
		\begin{align*}
			&f_2(\theta_2) = 0,\\
			&f(\theta_2) > 0.
		\end{align*}
		
		Therefore, let 
		
		\[g_2(\lambda) = f(\lambda) - f_2(\lambda) + \frac{1}{4}(5t^2 - 6tn + 16ts - 4t + n^2 - 8ns + 4n + 12s^2 - 12s)f_2(\lambda).\]
		We can prove that $f(\theta_2) > 0$ by proving $g_2(\theta_2) > 0$. Given below is the explicit formula for $g_2(\lambda)$, which is a linear function with respect to $\lambda$:
		
		\begin{align*}
			g_2(\lambda) = &\frac{1}{8}\lambda(n-t-2s)(8tn^2 - 2t^3 - 4t^2n - 20t^2s + 11t^2 + 8tns - 10tn - 38ts^2 + 48ts - 6t - 2n^3\\& + 8n^2s - 5n^2 + 14ns^2 - 32ns + 14n - 20s^3 + 40s^2 - 12s - 4)+\frac{1}{64}(n - t - 2s)(n^5 - 5t^5 + t^4n\\
			& - 6t^4s + 14t^4 + 10t^3n^2 - 12t^3n + 64t^3s^2 + 16t^3s - 8t^3 - 2t^2n^3 + 12t^2n^2s - 56t^2n^2 - 64t^2ns^2\\
			& - 72t^2ns + 88t^2n + 256t^2s^3 - 264t^2s^2 + 56t^2s - 40t^2 - 5tn^4 + 60tn^3 - 56tn^2 - 128tns^3\\
			& + 48tns^2 + 160tns - 16tn + 320ts^4 - 688ts^3 + 464ts^2 - 176ts + 32t - 6n^4s - 6n^4 + 56n^3s\\
			& - 24n^3 + 56n^2s^2 - 184n^2s + 56n^2 - 64ns^4 + 112ns^3 - 48ns^2 + 112ns - 32n + 128s^5 - 416s^4\\
			& + 480s^3 - 256s^2 + 32s).
		\end{align*}
		Since $n-t\ge 2s$, we analyze the sign of 
		
		\begin{align*}
			\tilde{g_2}(\lambda) = &\frac{1}{8}\lambda(8tn^2 - 2t^3 - 4t^2n - 20t^2s + 11t^2 + 8tns - 10tn - 38ts^2 + 48ts - 6t - 2n^3 + 8n^2s - 5n^2\\& + 14ns^2 - 32ns + 14n - 20s^3 + 40s^2 - 12s - 4) +\frac{1}{64}(n^5 - 5t^5 + t^4n - 6t^4s + 14t^4 + 10t^3n^2\\& - 12t^3n + 64t^3s^2 + 16t^3s - 8t^3 - 2t^2n^3 + 12t^2n^2s - 56t^2n^2 - 64t^2ns^2 - 72t^2ns + 88t^2n\\& + 256t^2s^3 - 264t^2s^2 + 56t^2s - 40t^2 - 5tn^4 + 60tn^3 - 56tn^2 - 128tns^3 + 48tns^2 + 160tns\\& - 16tn + 320ts^4 - 688ts^3 + 464ts^2 - 176ts + 32t - 6n^4s - 6n^4 + 56n^3s - 24n^3 + 56n^2s^2\\& - 184n^2s + 56n^2 - 64ns^4 + 112ns^3 - 48ns^2 + 112ns - 32n + 128s^5 - 416s^4 + 480s^3 - 256s^2\\& + 32s).
		\end{align*}
		
		Let $a = \frac{n-t}{2}, b = \frac{n+t}{2}$, then
		
		\begin{align*}
			\tilde{g_2}(\lambda) = &\lambda(- 3a^3 - 4a^2b - 5a^2s + 4a^2 + 3ab^2 + 14abs - 8ab + 13as^2 - 20as + 5a - b^2s - b^2 - 6bs^2 + 4bs \\&+ 2b - 5s^3 + 10s^2 - 3s - 1) + (- 3a^4 + 3a^3b^2 - 7a^3b - 4a^3s^2 - a^3s + 4a^3 - 2a^2b^3 - 3a^2b^2s + 5a^2b^2\\& + 8a^2bs^2 + 9a^2bs - 4a^2b + 12a^2s^3 - 8a^2s^2 - 9a^2s + a^2 + 2ab^3 - 4ab^2s^2 + 6ab^2s - 6ab^2 - 16abs^3\\& + 20abs^2 - 15abs + 6ab - 12as^4 + 25as^3 - 16as^2 + 9as - 2a + 4b^2s^3 - 5b^2s^2 + b^2s + 8bs^4 - 18bs^3\\& + 13bs^2 - 2bs + 4s^5 - 13s^4 + 15s^3 - 8s^2 + s).
		\end{align*}
		$\theta_2$ can be computed from the quotient matrix
		
		\[\begin{pmatrix}n\\2\end{pmatrix}Q_{2} = \begin{pmatrix}
			(a+b-1)(a-1) & ab\\
			a^2 & a(b-1)
		\end{pmatrix}\]
		to be
		\[\theta_2 = \frac{a^2}{2} + \frac{1}{2} + ab - \frac{b}{2} - \frac{3a}{2} + \frac12\sqrt{a^4 + 4a^3b - 2a^3 - 2a^2b + 3a^2 + 2ab - 2a + b^2 - 2b + 1}.\]
		
		Since $n\le 3t$, we have $b\ge 2a$. That implies
		
		\[\theta_2\ge\theta_2' = \frac{a^2}{2} + \frac{1}{2} + ab - \frac{b}{2} - \frac{3a}{2} + \frac12(3a^2 - a).\]
		
		Indeed
		
		\begin{align*}
			&(a^4 + 4a^3b - 2a^3 - 2a^2b + 3a^2 + 2ab - 2a + b^2 - 2b + 1) - (3a^2 - a)^2\\
			=&- 8a^4 + 4a^3b + 4a^3 - 2a^2b + 2a^2 + 2ab - 2a + b^2 - 2b + 1:=h(b).
		\end{align*}
		Since $a,b\ge 1$, $h'(b) = 4a^3 - 2a^2 + 2a + 2b - 2 > 0$, therefore $h(b) \ge h(2a) = 10a^2 - 6a + 1 > 0$. We can prove that $\tilde{g_2}(\theta_2) \ge \tilde{g_2}(\theta_2') > 0$. First, the coefficient $c_1$ of $\lambda$ in $\tilde{g_2}(\lambda)$ is positive
		
		\begin{align*}
			&c_1 = - 3a^3 - 4a^2b - 5a^2s + 4a^2 + 3ab^2 + 14abs - 8ab + 13as^2 - 20as + 5a - b^2s - b^2 - 6bs^2 + 4bs + 2b\\& - 5s^3 + 10s^2 - 3s - 1,\\
			&\frac{\partial c_1}{b} = 4s - 2b - 8a + 6ab + 14as - 2bs - 4a^2 - 6s^2 + 2,\\
			&\frac{\partial^2c_1}{\partial b^2} = 6a - 2s - 2 > 0
		\end{align*}
		therefore
		\begin{align*}
			&\frac{\partial c_1}{b}\ge\frac{\partial c_1}{b}(2a) = 8a^2 + 10as - 12a - 6s^2 + 4s + 2 = (6a^2 - 6s^2) + (2a^2 + 10as - 12s) + 2 \ge 0,\\
			&c_1\ge c_1(2a) = a^3 + 19a^2s - 16a^2 + as^2 - 12as + 9a - 5s^3 + 10s^2 - 3s - 1.
		\end{align*}
		
		Consider the right hand side of the above inequality a function of $a$, since $a\ge s\ge 1$
		\begin{align*}
			&RHS'(a) = 3a^2 + 38as - 32a + s^2 - 12s + 9,\\
			&RHS''(a) = 6a + 38s - 32 > 0\\
			\Rightarrow &RHS'(a) \ge RHS'(s) = 42s^2 - 44s + 9 > 0,\\
			\Rightarrow &RHS(a) \ge RHS(s) = 16s^3 - 18s^2 + 6s - 1 > 0.
		\end{align*}
		
		It remains to show that $\tilde{g_2}(\theta_2') > 0$.
		
		\begin{align*}
			\tilde{g_2}(\theta_2') = &- 6a^5 - 11a^4b - 10a^4s + 11a^4 + 5a^3b^2 + 23a^3bs - \frac{19a^3b}2 + 22a^3s^2 - 31a^3s + \frac{9a^3}2 + a^2b^3 + 9a^2b^2s\\& - 9a^2b^2 + 9a^2bs^2 - \frac{57a^2bs}2 + 17a^2b + 2a^2s^3 - 14a^2s^2 + \frac{45a^2s}{2} - 9a^2 - ab^3s - \frac{ab^3}2 - 10ab^2s^2\\& + 5ab^2s + \frac{7ab^2}2 - 21abs^3 + \frac{71abs^2}2 - 9abs - \frac{11ab}2 - 12as^4 + 35as^3 - \frac{59as^2}{2} + 5as + \frac{5a}2 + \frac{b^3s}2\\& + \frac{b^3}2 + 4b^2s^3 - 2b^2s^2 - \frac{3b^2s}2 - \frac{3b^2}2 + 8bs^4 - \frac{31bs^3}2 + 5bs^2 + \frac{3bs}2 + \frac{3b}2 + 4s^5 - 13s^4 + \frac{25s^3}2 - 3s^2\\& - \frac{s}2 - \frac{1}2,\\
			\frac{\partial \tilde{g_2}}{\partial b} =& - 11a^4 + 10a^3b + 23a^3s - \frac{19a^3}2 + 3a^2b^2 + 18a^2bs - 18a^2b + 9a^2s^2 - \frac{57a^2s}2 + 17a^2 - 3ab^2s - \frac{3ab^2}2\\& - 20abs^2 + 10abs + 7ab - 21as^3 + \frac{71as^2}2 - 9as - \frac{11a}2 + \frac{3b^2s}2 + \frac{3b^2}2 + 8bs^3 - 4bs^2 - 3bs - 3b\\& + 8s^4 - \frac{31s^3}2 + 5s^2 + \frac{3s}2 + \frac32,\\
			\frac{\partial^2 \tilde{g_2}}{\partial b^2} =& 7a + 3b - 3s - 3ab + 10as + 3bs + 6a^2b - 20as^2 + 18a^2s - 18a^2 + 10a^3 - 4s^2 + 8s^3 - 6abs - 3,\\
			\frac{\partial^3 \tilde{g_2}}{\partial b^3} = &3s - 3a - 6as + 6a^2 + 3 > 0.
		\end{align*}
		
		Therefore
		
		\begin{align*}
			\frac{\partial^2 \tilde{g_2}}{\partial b^2} \ge \frac{\partial^2 \tilde{g_2}}{\partial b^2}(2a) = 22a^3 + 6a^2s - 24a^2 - 20as^2 + 16as + 13a + 8s^3 - 4s^2 - 3s - 3.
		\end{align*}
		Consider the right hand side a function of $a$, since $a\ge s\ge 1$
		\begin{align*}
			&RHS'(a) = 66a^2 + 12as - 48a - 20s^2 + 16s + 13 = (66a^2 - 48a - 8s^2) + (12as - 12s^2) + 16s + 13 > 0,\\
			\Rightarrow &RHS(a)\ge RHS(s) = 16s^3 - 12s^2 + 10s - 3 > 0.
		\end{align*}
		
		Now, since $\frac{\partial^2 \tilde{g_2}}{\partial b^2} > 0$
		\begin{align*}
			\frac{\partial \tilde{g_2}}{\partial b} \ge \frac{\partial \tilde{g_2}}{\partial b}(2a) = &21a^4 + 47a^3s - \frac{103a^3}2 - 31a^2s^2 - \frac{5a^2s}2 + 37a^2 - 5as^3 + \frac{55as^2}2 - 15as - \frac{23a}2 + 8s^4\\
			&- \frac{31s^3}2 + 5s^2 + \frac{3s}2 + \frac32.
		\end{align*}
		Consider the right hand side a function of $a$
		\begin{align*}
			&RHS'(a) = 84a^3 + 141a^2s - \frac{309a^2}2 - 62as^2 - 5as + 74a - 5s^3 + \frac{55s^2}2 - 15s - \frac{23}2,\\
			&RHS''(a) = 252a^2 + 282as - 309a - 62s^2 - 5s + 74 = (252a^2 - 252a) + (282as - 57a - 62s^2 - 5s)\\& + 74 > 0.\\
			\Rightarrow&RHS'(a)\ge RHS'(s) = 158s^3 - 132s^2 + 59s - 23/2 > 0,\\
			\Rightarrow &RHS(a) \ge RHS(s) = 40s^4 - 42s^3 + 27s^2 - 10s + \frac32 > 0 \text{ since }s\ge 1.
		\end{align*}
		
		Finally, since $\frac{\partial \tilde{g_2}}{\partial b} > 0$
		
		\begin{align*}
			\tilde{g_2}\ge \tilde{g_2}(2a) = &64a^4s - 48a^4 - 64a^3s + \frac{113a^3}2 - 24a^2s^3 + 49a^2s^2 - \frac{3a^2s}2 - 26a^2 + 4as^4 + 4as^3 - \frac{39as^2}2\\
			&+ 8as + \frac{11a}2 + 4s^5 - 13s^4 + \frac{25s^3}2 - 3s^2 - \frac{s}2 - \frac12.
		\end{align*}
		Consider the right hand side a function of $a$
		\begin{align*}
			&RHS'(a) = 256a^3s - 192a^3 - 192a^2s + \frac{339a^2}2 - 48as^3 + 98as^2 - 3as - 52a + 4s^4 + 4s^3 - \frac{39s^2}2 + 8s\\& + \frac{11}{2}\\
			&RHS''(a) = 768a^2s - 576a^2 - 384as + 339a - 48s^3 + 98s^2 - 3s - 52\\
			&RHS^{(3)}(a) = 1536as - 384s - 1152a + 339 > 0
		\end{align*}
		therefore
		\begin{align*}
			&RHS''(a)\ge RHS''(s) = 720s^3 - 862s^2 + 336s - 52 > 0,\\
			&RHS'(a) \ge RHS'(s) = 212s^4 - 282s^3 + 147s^2 - 44s + \frac{11}2 > 0,\\
			&RHS(a) \ge RHS(s) = 48s^5 - 72s^4 + 48s^3 - 21s^2 + 5s - \frac12 > 0.
		\end{align*}
		Thus, we have the desired proof of Claim \ref{extremalclaim2}
	\end{proof}
	
	We have proved that, for graphs with fixed matching number $\frac{n-t}{2}$, the extremal graph is either $G_1 = K_1\vee({K_{n-t-1}\cup \overline{K_{t}}})$ or $G_2 = K_{\frac{n-t}{2}}\vee{\overline{K_{\frac{n+t}{2}}}}$. We now compare the spectral radii of $G_1$ and $G_2$ to $\frac{4\alpha'(G)}{n}$.
	
	\begin{claim}
		\label{claim:bound1} $$\rho(A_{\overline{\kappa}}(G_1)) < \frac{2(n-t)}{n}.$$
	\end{claim}
	
	\begin{proof}
		From the quotient matrix given in the proof of Claim \ref{extremalclaim1}, we can compute
		\begin{align*}
			&\rho(A_{\overline{\kappa}}(G_1)) = \frac{1}{\begin{pmatrix}
					n\\2
			\end{pmatrix}}\left(\frac{3t}2 - n  - nt + \frac{n^2}2 + \frac{t^2}2 \right.\\&\left.+\frac{1}{2}\sqrt{n^4 - 4n^3t - 4n^3 + 6n^2t^2 + 10n^2t + 8n^2 - 4nt^3 - 8nt^2 - 8nt - 8n + t^4 + 2t^3 + t^2 + 4t + 4}\vphantom{\frac11}\right).
		\end{align*}
		Therefore
		\begin{align*}
			&\rho(A_{\overline{\kappa}}(G_1))< \frac{2(n-t)}{n}\\
			\Leftrightarrow &\sqrt{n^4 - 4n^3t - 4n^3 + 6n^2t^2 + 10n^2t + 8n^2 - 4nt^3 - 8nt^2 - 8nt - 8n + t^4 + 2t^3 + t^2 + 4t + 4}\\
			<&2\left(\begin{pmatrix}
				n\\2
			\end{pmatrix}\frac{2(n-t)}{n} - \left(\frac{3t}2 - n  - nt + \frac{n^2}2 + \frac{t^2}2\right)\right) = n^2 - t^2 - t\\
			\Leftrightarrow &0 < 4n^3t + 4n^3 - 8n^2t^2 - 12n^2t - 8n^2 + 4nt^3 + 8nt^2 + 8nt + 8n - 4t - 4.
		\end{align*}
		
		Consider the right hand side of the above inequality a function of $n$, since $n \ge 3t + 2$
		\begin{align*}
			&RHS'(n) = 12n^2t + 12n^2 - 16nt^2 - 24nt - 16n + 4t^3 + 8t^2 + 8t + 8,\\
			&RHS''(n) = 24n - 24t + 24nt - 16t^2 - 16 > 0
		\end{align*}
		\begin{align*}
			\Rightarrow & RHS'(n) \ge RHS'(3t + 2) = 64t^3 + 156t^2 + 104t + 24 > 0,\\
			&RHS(n) \ge RHS(3t + 2) = 48t^4 + 152t^3 + 152t^2 + 68t + 12 > 0
		\end{align*}
		which gives the proof of Claim \ref{claim:bound1}.
	\end{proof}
	
	\begin{claim}
		\label{claim:bound2} $$\rho(A_{\overline{\kappa}}(G_2)) < \frac{2(n-t)}{n}.$$
	\end{claim}
	
	\begin{proof}
		From the quotient matrix given in the proof of Claim \ref{extremalclaim2}, we can compute
		\begin{align*}
			&\rho(A_{\overline{\kappa}}(G_2)) = \frac{1}{\begin{pmatrix}
					n\\2
			\end{pmatrix}}\left(\frac{t}2 - n  - \frac{nt}4 + \frac{3n^2}8 - \frac{t^2}8 + \frac12 + \right.\\&\left.+\frac18\sqrt{5n^4 - 12n^3t - 8n^3 + 6n^2t^2 + 16n^2t + 24n^2 + 4nt^3 - 8nt^2 - 16nt - 32n - 3t^4 + 8t^2 + 16}\vphantom{\frac11}\right).
		\end{align*}
		Therefore
		\begin{align*}
			&\rho(A_{\overline{\kappa}}(G_2))< \frac{2(n-t)}{n}\\
			\Leftrightarrow &\sqrt{5n^4 - 12n^3t - 8n^3 + 6n^2t^2 + 16n^2t + 24n^2 + 4nt^3 - 8nt^2 - 16nt - 32n - 3t^4 + 8t^2 + 16}\\
			<&8\left(\begin{pmatrix}
				n\\2
			\end{pmatrix}\frac{2(n-t)}{n} - \left(\frac{t}2 - n  - \frac{nt}4 + \frac{3n^2}8 - \frac{t^2}8 + \frac12\right)\right) = 5n^2 - 6nt + t^2 + 4t - 4\\
			\Leftrightarrow &0 < 20n^4 - 48n^3t + 8n^3 + 40n^2t^2 + 24n^2t - 64n^2 - 16nt^3 - 40nt^2 + 64nt + 32n + 4t^4 + 8t^3 - 32t.
		\end{align*}
		
		Consider the right hand side of the above inequality a function of $n$, since $n \ge t + 2$
		\begin{align*}
			&RHS'(n) = 80n^3 - 144n^2t + 24n^2 + 80nt^2 + 48nt - 128n - 16t^3 - 40t^2 + 64t + 32,\\
			&RHS''(n) = 240n^2 - 288nt + 48n + 80t^2 + 48t - 128,\\
			&RHS^{(3}(n) = 480n - 288t + 48 > 0.
		\end{align*}
		\begin{align*}
			\Rightarrow & RHS''(n) \ge RHS''(t + 2) = 32t^2 + 480t + 928 > 0,\\
			&RHS'(n) \ge RHS'(t + 2) = 96t^2 + 512t + 512 > 0,\\
			&RHS(n) \ge RHS(t + 2) = 128t^2 + 320t + 192 > 0
		\end{align*}
		thus, the proof of Claim \ref{claim:bound2} is given.
	\end{proof}
	
	The complete proof of Theorem \ref{main} follows directly from Claims \ref{claim1}, \ref{mainclaim}, \ref{extremalclaim1}, \ref{extremalclaim2}, \ref{claim:bound1} and \ref{claim:bound2}. From Claim \ref{claim1}, equality is achieved when $G$ is a complete graph and $\alpha'(G) = \frac{n-1}{2}$, implying $n$ is odd. In the remaining case, the extremal graphs are $G_1$ and $G_2$ whose spectral radii are strictly smaller than $\frac{4\alpha'(G)}{n}$. Thus the proof of Theorem \ref{main} is done.\qquad \qquad \qquad $\qedsymbol$

	\section{Proof of theorem \ref{bipartitethm}}
	
	\quad\\
	\textit{Proof of theorem \ref{bipartitethm}.} We now turn our attention to bipartite graphs, in which we can improve the upper bound in Theorem \ref{main}.
	
	\begin{claim}\label{claim:bipartite1}
		$\rho(A_{\overline{\kappa}}(K_{k,n-k})) \le \frac{(n-k)(4k - 2)}{n(n-1)}$.
	\end{claim}
	\begin{proof}
		Without loss of generality, suppose that $k\le n - k$ or $2k \le n$. Partition the vertices according to the partite sets, we have the quotient matrix
		
		\[\begin{pmatrix}
			n\\2
		\end{pmatrix}\tilde{Q} =\begin{pmatrix}
			(n-k)(k-1)&k(n-k)\\
			k^2 & k(n-k-1)
		\end{pmatrix}.\]
		
		The spectral radius of $A_{\overline{\kappa}}(K_{k,n-k})$ can be computed
		
		\[\rho(A_{\overline{\kappa}}(K_{k,n-k})) = \frac{1}{n(n-1)}\left(\sqrt{n^2 + 4nk^3 - 4nk-4k^4+4k^2} + 2nk -n - 2k^2\right)\]
		and analyzed as follows
		
		\begin{align*}
			&\sqrt{n^2 + 4nk^3 - 4nk-4k^4+4k^2} + 2nk -n - 2k^2\\
			=& \frac{4nk^3 - 4nk - 4k^4+4k^2}{\sqrt{(n-2k)^2 + 4nk^3 -4k^4} + n} + 2nk - 2k^2\\
			\le& \frac{4nk^3 - 4nk - 4k^4+4k^2}{\sqrt{8k^4 -4k^4} + 2k} + 2nk - 2k^2 \hspace{3cm}{(*)}\\
			=& \frac{4nk(k-1)(k+1) - 4k^2(k-1)(k+1)}{2k^2 + 2k} + 2nk - 2k^2\\
			=& 2n(k-1) - 2k(k-1) + 2nk - 2k^2\\
			=& 2(n-k)(k-1) + 2k(n-k)\\
			=& (n-k)(4k-2),
		\end{align*}
		hence, Claim \ref{claim:bipartite1} is proved.
	\end{proof}
	
	The next step is to show that $K_{k, n-k}$ is the extremal graph with the maximum spectral radius among $n$-vertex connected bipartite graphs with the matching number $k$.
	
	\begin{claim}\label{claim:bipartite2}
		Let $G$ be an $n$-vertex connected bipartite graph with the matching number $k$. Then, $\rho(A_{\overline{\kappa}}(G))\le \rho(A_{\overline{\kappa}}(K_{k, n -k}))$.
	\end{claim}
	\begin{proof}
		Denote by $X, Y$ the two partite sets of $G$, and assume that $|X|\le |Y|$. If $G$ has a matching that covers every vertex of $X$, then $|X| = k$ and $E(G)\subset E(K_{k,n-k})$. By Theorem~\ref{PerronCor}, $\rho(A_{\overline{\kappa}}((G))\le\rho(A_{\overline{\kappa}}(K_{k,n-k}))$.
		
		If $G$ does not have a matching that covers every vertex of $X$, by Hall's marriage theorem, there exists a subset $S$ of $X$ such that $\alpha'(G) = |X| - |S| + |N(S)| = k$, where $N(S)$ denotes the neighbors of $S$.
		
		Consider the graph $G^*$ (see Figure \ref{fig:bipartite})  constructed by adding every edge possible between $S$ and $N(S)$, $X-S$ and $N(S)$, $X-S$ and $Y - N(S)$.
		
		\begin{figure}[!h]
			\centering
			\begin{tikzpicture}
				\draw (0,0) ellipse (1.5cm and 0.8cm);
				\node at (0,0) (S) {$S$};
				\draw (0,-5) ellipse (1.5cm and 0.8cm);
				\node at (0,-5) (NS) {$N(S)$};
				\draw (5,0) ellipse (1.5cm and 0.8cm);
				\node at (5,0) (X_S) {$X-S$};
				\draw (5,-5) ellipse (1.5cm and 0.8cm);
				\node at (5,-5) (Y_NS) {$Y-N(S)$};
				\draw[-] (-0.5, -0.5) -- (-0.5, -4.5);
				\draw[-] (-0.5, -0.5) -- (0, -4.5);
				\draw[-] (-0.5, -0.5) -- (0.5, -4.5);
				\draw[-] (0, -0.5) -- (0, -4.5);
				\draw[-] (0, -0.5) -- (0.5, -4.5);
				\draw[-] (0, -0.5) -- (-0.5, -4.5);
				\draw[-] (0.5, -0.5) -- (0.5, -4.5);
				\draw[-] (0.5, -0.5) -- (0, -4.5);
				\draw[-] (0.5, -0.5) -- (-0.5, -4.5);
				
				\draw[-] (5, -0.5) -- (-0.5, -4.5);
				\draw[-] (5, -0.5) -- (0.5, -4.5);
				\draw[-] (5, -0.5) -- (0, -4.5);
				\draw[-] (5.5, -0.5) -- (0, -4.5);
				\draw[-] (5.5, -0.5) -- (0.5, -4.5);
				\draw[-] (5.5, -0.5) -- (-0.5, -4.5);
				\draw[-] (4.5, -0.5) -- (0.5, -4.5);
				\draw[-] (4.5, -0.5) -- (0, -4.5);
				\draw[-] (4.5, -0.5) -- (-0.5, -4.5);
				
				\draw[-] (5, -0.5) -- (4.5, -4.5);
				\draw[-] (5, -0.5) -- (5.5, -4.5);
				\draw[-] (5, -0.5) -- (5, -4.5);
				\draw[-] (5.5, -0.5) -- (4.5, -4.5);
				\draw[-] (5.5, -0.5) -- (5.5, -4.5);
				\draw[-] (5.5, -0.5) -- (5, -4.5);
				\draw[-] (4.5, -0.5) -- (4.5, -4.5);
				\draw[-] (4.5, -0.5) -- (5.5, -4.5);
				\draw[-] (4.5, -0.5) -- (5, -4.5);
			\end{tikzpicture}
			\caption{The bipartite graph $G^*$}
			\label{fig:bipartite}
		\end{figure}
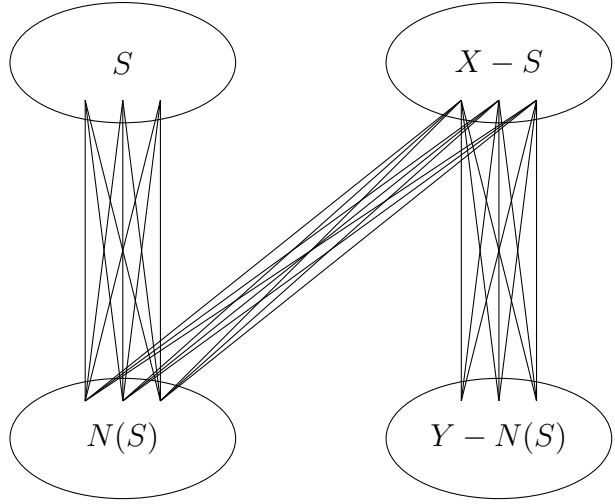
		
		By construction, $\alpha'(G^*)  = \alpha'(G) = k$ and $\rho(A_{\overline{\kappa}}(G^*))\ge \rho(A_{\overline{\kappa}}(G))$. Let $|X| = x, |Y| = y, |S| = s$ and $|N(S)| = n_s$, the quotient matrix of $G^*$ corresponding to the partition $\{S, N(S), X - S, Y - N(S)\}$ is then
		
		\[\hspace*{-0.5cm}\begin{pmatrix}
			n\\2
		\end{pmatrix}\tilde{Q_1} =\begin{pmatrix}
			(s-1)n_s&n_s^2&(x-s)n_s&(y-n_s)\min\{n_s,x-s\}\\
			sn_s & (n_s - 1)x & (x-s)(n_s+x-s-1) & (y-n_s)(x-s)\\
			sn_s & n_s(x-s+n_s-1) & (x-s-1)y & (y-n_s)(x-s)\\
			s\min\{n_s,x-s\}&n_s(x-s)&(x-s)^2&(y-n_s-1)(x-s)
		\end{pmatrix}.\]
		
		Since $\tilde{Q_1}$ depends on the value of $\min\{n_s, x-s\}$, we consider two cases.
		
		\quad
		
		\noindent\textit{Case 1:} If $n_s \le x-s$, then
		\[\begin{pmatrix}
			n\\2
		\end{pmatrix}\tilde{Q_1} = \begin{pmatrix}
			(s-1)n_s&n_s^2&(x-s)n_s&(y-n_s)n_s\\
			sn_s & (n_s - 1)x & (x-s)(n_s+x-s-1) & (y-n_s)(x-s)\\
			sn_s & n_s(x-s+n_s-1) & (x-s-1)y & (y-n_s)(x-s)\\
			sn_s&n_s(x-s)&(x-s)^2&(y-n_s-1)(x-s)
		\end{pmatrix}.\]
		
		Consider the graph $G^{**}$ obtained by moving one vertex from $S$ to $Y-N(S)$ and its quotient matrix
		
		\[\begin{pmatrix}
			n\\2
		\end{pmatrix}\tilde{Q_2} =\begin{pmatrix}
			(s-2)n_s&n_s^2&(x-s)n_s&(y-n_s+1)n_s\\
			(s-1)n_s & (n_s - 1)(x-1) & (x-s)(n_s+x-s-1) & (y-n_s+1)(x-s)\\
			(s-1)n_s & n_s(x-s+n_s-1) & (x-s-1)(y+1) & (y-n_s+1)(x-s)\\
			(s-1)n_s&n_s(x-s)&(x-s)^2&(y-n_s)(x-s)
		\end{pmatrix},\]
		then
		\[\begin{pmatrix}
			n\\2
		\end{pmatrix}(\tilde{Q_2} - \tilde{Q_1}) = \begin{pmatrix}
			-n_s&0&0&n_s\\
			-n_s&-(n_s-1)&0&x-s\\
			-n_s&0&x-s-1&x-s\\
			-n_s&0&0&x-s
		\end{pmatrix}.\]
		
		Let $u = [u_1, u_2, u_3, u_4]^T > 0$ be the Perron vector corresponding to the spectral radius $\rho$ of $\begin{pmatrix}
			n\\2
		\end{pmatrix}\tilde{Q_1}$. We compare the entries of $u$, namely $u_1$ and $u_4$. If $n_s = x-s$ then
		\[(s-1)n_su_1 + n_s^2u_2+n_s^2u_3 + (y-n_s)n_su_4 = \rho u_1,\]
		\[sn_su_1 + n_s^2u_2+n_s^2u_3 + (y-n_s-1)n_su_4 = \rho  u_4\]
		
		\[\Rightarrow\rho(u_1 - u_4) = -n_s(u_1 - u_4) \Rightarrow u_1 = u_4.\]
		Else if $n_s\le x-s-1$ then
		\begin{align*}
			(y-n_s-1)(x-s) &\ge (y-n_s-1)(n_s+1)\\
			&= (y-n_s)n_s + (y-n_s-1) - n_s\\
			&\ge (y-n_s)n_s
		\end{align*}
		(since $x\le y, n_s\le s \Rightarrow y - n_s \ge x-s\ge n_s + 1$). From here, we have
		\begin{align*}
			&\rho u_1 = (s-1)n_su_1 +n_s^2u_2 + (x-s)n_su_3 + (y-n_s)n_su_4\\
			&\le sn_su_1 + n_s(x-s)u_2 + (x-s)^2u_3 + (y-n_s-1)(x-s)u_4 = \rho u_4,
		\end{align*}
		therefore $u_1 \le u_4$. Next, we compare $u_2$ and $u_3$.
		
		\begin{align*}
			&\rho(u_2 - u_3) = [(n_s-1)x - n_s(x-s+n_s-1)]u_2\\
			&+ [(x-s)(n_s+x-s-1) - (x-s-1)y]u_3\\
			\Rightarrow &\frac{u_2}{u_3} = \frac{\rho - (x-s-1)y+(x-s)(n_s+x-s-1)}{\rho - (n_s-1)x + n_s(n_s+x-s-1)}.
		\end{align*}
		
		Subtracting the numerator by the denominator gives
		
		\begin{align*}
			&(x-s)(n_s+x-s-1) - (x-s-1)y + (n_s-1)x - n_s(n_s+x-s-1)\\
			=&(x-s - n_s)(n_s+x-s-1) - (x-s-1)y + (n_s-1)x\\
			=&(x-s - n_s)(n_s+x-s-1) - (x-s-n_s)y - (n_s-1)y + (n_s-1)x\\
			=&(x-s-n_s)(n_s+x-s-y-1) - (n_s-1)(y-x) \le 0,
		\end{align*}
		therefore $u_2 \le u_3$.
		
		Next, we prove that $\begin{pmatrix}
			n\\2
		\end{pmatrix}u^T(\tilde{Q_2}-\tilde{Q_1})u\ge 0$. Indeed,
		
		\begin{align*}
			\begin{pmatrix}
				n\\2
			\end{pmatrix}u^T(\tilde{Q_2}-\tilde{Q_1})u = &-n_su_1^2 - n_su_1u_2 - n_su_1u_3-(n_s-1)u_2^2 + (x-s)u_2u_4 + (x-s-1)u_3^2\\
			&+ (x-s)u_3u_4+ (x-s)u_4^2.
		\end{align*}
		Combining $u_1\le u_4, u_2\le u_3$ with the assumption $n_s\le x-s$, we have
		
		\begin{align*}
			&n_su_1^2\le (x-s)u_4^2,\\
			&n_su_1u_2\le (x-s)u_2u_4,\\
			&n_su_1u_3\le (x-s)u_3u_4,\\
			&(n_s-1)u_2^2\le (x-s-1)u_3^2
		\end{align*}
		therefore, $u^T(\tilde{Q_2}-\tilde{Q_1})u\ge 0$ and \[\rho\left(\begin{pmatrix}
			n\\2
		\end{pmatrix}\tilde{Q_2}\right) \ge \begin{pmatrix}
			n\\2
		\end{pmatrix}u^T\tilde{Q_2}u \ge \begin{pmatrix}
			n\\2
		\end{pmatrix}u^T\tilde{Q_1}u = \rho\left(\begin{pmatrix}
			n\\2
		\end{pmatrix}\tilde{Q_1}\right).\]
		
		Each time we move a vertex from $S$ to $Y-N(S)$, we increase the spectral radius, and we can keep doing so without changing the matching number so long as $|S|\ge|N(S)|$ after the vertex is moved. When $|S| = |N(S)|$, clearly $|X| = k$ and $E(G) \subset E(K_{k,n-k})$, so it follows that $\rho(A_{\overline{\kappa}}(G))\le\rho(A_{\overline{\kappa}}(K_{k,n-k}))$.
		
		\quad
		
		\noindent\textit{Case 2:} If $n_s > x - s$, the quotient matrix of $G^*$ now becomes
		
		\[\begin{pmatrix}
			n\\2
		\end{pmatrix}\tilde{Q_1} = \begin{pmatrix}
			(s-1)n_s&n_s^2&(x-s)n_s&(y-n_s)(x-s)\\
			sn_s & (n_s - 1)x & (x-s)(n_s+x-s-1) & (y-n_s)(x-s)\\
			sn_s & n_s(x-s+n_s-1) & (x-s-1)y & (y-n_s)(x-s)\\
			s(x-s)&n_s(x-s)&(x-s)^2&(y-n_s-1)(x-s)
		\end{pmatrix}.\]
		
		By Theorem \ref{quotient} and Theorem \ref{Gctheo}, the spectral radius of $\tilde{Q_1}$ is bounded by the largest row sum of $\tilde{Q_1}$. We consider each row sum:
		
		\begin{enumerate}
			
			\item $(s-1)n_s+n_s^2+(x-s)n_s+(y-n_s)(x-s)$
			
			$\le (s-1)n_s+n_s^2+(x-s)n_s+(y-n_s)n_s$
			
			$ = n_s(s-1+n_s+x-s+y-n_s)$
			
			$ = n_s(n-1) \le k(n-1)$.
			
			\item $sn_s+(n_s-1)x+(x-s)(n_s+x-s-1)+(y-n_s)(x-s)$
			
			$= sn_s+(n_s-1)x+(x-s)(y+x-s-1)$
			
			$\le sn_s+(n_s-1)x+(x-s)(n-1)$
			
			$\le (y-1)n_s+n_sx+(x-s)(n-1)$ (since $G$ is connected, $x-s \ge 1 \Rightarrow y - s\ge 1$)
			
			$= n_s(n-1) + (x-s)(n-1)$
			
			$= k(n-1)$.
			
			\item $sn_s+n_s(x-s+n_s-1)+(x-s-1)y+(y-n_s)(x-s)$
			
			$\le sn_s+n_s(x-s+n_s-1) + (x-s-1)y+(y-n_s)n_s$
			
			$= n_s(s +x-s+n_s-1 + y - n_s) + (x-s-1)y$
			
			$= n_s(n-1) + (x-s-1)y$
			
			$\le n_s(n-1) + (x-s)(n-1)$
			
			$= k (n-1)$.
			
			\item $s(x-s) + n_s(x-s) + (x-s)^2 + (y-n_s-1)(x-s)$
			
			$= (x-s)(s+n_s+x-s+y-n_s-1)$
			
			$= (x-s)(n-1) \le k(n-1)$.
		\end{enumerate}
		
		Hence, all row sums of $\tilde{Q_1}$ are bounded by $\frac{2k}{n}$, which we can prove is no greater than $\rho(A_{\overline{\kappa}}(K_{k,n-k}))$. Consider that
		
		\begin{align*}
			\rho(A_{\overline{\kappa}}(K_{k,n-k})) &= \frac{1}{n(n-1)}\left(\sqrt{n^2 + 4nk^3 - 4nk-4k^4+4k^2} + 2nk -n - 2k^2\right)\\
			&= \frac{1}{n(n-1)}\left(\sqrt{n^2 + 4nk^3 - 4nk-4k^4+4k^2} + 2(n-1)k - n - 2k^2 + 2k\right)\\
			&= \frac{2k}{n} + \frac{1}{n(n-1)}\left(\sqrt{n^2 + 4nk^3 - 4nk-4k^4+4k^2}- n - 2k^2 + 2k\right),
		\end{align*}
		therefore $\rho(A_{\overline{\kappa}}(K_{k,n-k})) \ge \frac{2k}{n}$ if $\sqrt{n^2 + 4nk^3 - 4nk-4k^4+4k^2}- n - 2k^2 + 2k \ge 0$. Consider the inequality
		
		\begin{align*}
			&\sqrt{n^2 + 4nk^3 - 4nk-4k^4+4k^2}\ge n-2k+2k^2\\
			\Leftrightarrow&n^2 + 4nk^3 - 4nk-4k^4+4k^2 \ge n^2 + 4k^2 + 4k^4 - 4nk + 4nk^2 - 8k^3\\
			\Leftrightarrow&4nk^3 - 8k^4\ge 4nk^2 - 8k^3\\
			\Leftrightarrow&(k^3 - k^2)(4n - 8k)\ge 0
		\end{align*}
		and since $k\ge 1$ and $n\ge 2k$, the inequality is true. The proof of Claim \ref{claim:bipartite2} is now complete.
		
	\end{proof}
	
	Combining Claim \ref{claim:bipartite1} and Claim \ref{claim:bipartite2}, we have the complete proof of Theorem \ref{bipartitethm}. Equality holds only when $G\cong K _{k,n-k}$ and $k = n - k$ because in order to have equality in $(*)$, we must have $n = 2k$. \hspace{13.6cm} $\qedsymbol$


\begin{thebibliography}{99}
		
		\bibitem{BOP}
		L. Beineke, O. Oellermann, R. Pippert, 
		The average connectivity of a graph, 
		\emph{Discrete Math.} {\bf 252} (2002), 31--45.
		
		\bibitem {BH} 
		A. Brouwer and W. Haemers, {\em Spectra of Graphs}, 245pp book (2011), available at {\tt http://www.win.tue.nl/~aeb/2WF02/spectra.pdf}.
		
		
		
		
		\bibitem{F}
		M. Fiedler, Algebraic connectivity of graphs. \emph{Czech.Math.J} {\bf 23} (1973),
		298-305. 
		
		\bibitem {GR} C. Godsil and G. Royle, {\em Algebraic Graph Theory}, Graduate Texts in Mathematics, 207. Springer-Verlag, New York, 2001.
		
		\bibitem{HJ}
		R.A. Horn and C.R. Johnson, \emph{Matrix Analysis}, Cambridge University Press 1999.
		
		\bibitem{KO}
		J. Kim and S. O, 
		Average connectivity and average edge-connectivity in graphs, 
		\emph{Discrete Math.}, {\bf 313} (2013), 2232--2238.
		
		\bibitem{LM}
		S. Li, and S. Miao, Characterizing ${\cal P}_{\ge 2}$-factor and ${\cal P}_{\ge 2}$-factor covered graphs with respect to the size or the spectral radius,  \emph{Discrete Math.} {\bf 344} (2021), 112588.
		
		
		\bibitem{O2}
		S. O, Spectral radius and matchings in graphs,
		\emph{Linear Algebra Appl.}, {\bf 614} (2021), 316--324.
		
		\bibitem{KOSS}
		M. Kim, S. O, W. Sim, and D. Shin, Matchings in graphs from the spectral radius, \emph{Linear Multilinear Algebra}. (published online 18th May, 2022) DOI: 10.1080/03081087.2022.2076799
		
		\bibitem{BT}
		W.T. Tutte, The factorization of linear graphs, \emph{J. London Math. Soc.} {\bf 22} (1947), 107--111.
		
		\bibitem{W} D.B. West, {\it Introduction to Graph Theory}, Prentice Hall, Inc., Upper Saddle River, NJ, 2001
		
		\bibitem{YYSX}
		L.H. You, M. Yang, W. So, and W. G. Xi, On the spectrum of an equitable matrix and its application, \emph{Linear Algebra Appl.} {\bf 577} (2019) 21--40.
		
		\bibitem{Z}
		W. Zhang, The maximum spectral radius of $t$-connected graphs with bounded matching number, \emph{Discrete Math.} {\bf 345} (2022), 112775.
		
		\bibitem{ZHW}
		Y. Zhao, X. Huang, and Z. Wang, The $A_{\alpha}$-spectral radius and perfect matchings of graphs, \emph{Linear Algebra Appl.} {\bf 631} (2021), 143--155.
		
		\bibitem{ZL}
		Y. Zhang and H. Lin, Perfect matching and distance spectral radius in graphs and bipartite graphs, \emph{Discrete Appl. Math.} {\bf 304} (2021), 315--322.
		
		
	\end{thebibliography}
\end{document}